\documentclass[10pt,reqno]{amsart}

\usepackage[T1]{fontenc}
\usepackage[margin=1in]{geometry}
\usepackage{amsmath,amssymb,amsthm,mathtools,microtype,booktabs,tabularx,cite}
\usepackage[
	colorlinks=true,
	linkcolor=blue,
	citecolor=blue,
	urlcolor=blue,
	pdfauthor={Marcus Appleby, Steven T. Flammia, Gene S. Kopp},
	pdftitle={The twisted convolution identity and ghost r-SICs from finite quantum dilogarithms},
]{hyperref}
\usepackage{autonum} 

\newtheorem{theorem}{Theorem}[section]
\newtheorem{proposition}[theorem]{Proposition}
\newtheorem{lemma}[theorem]{Lemma}

\theoremstyle{definition}
\newtheorem{definition}[theorem]{Definition}
\theoremstyle{remark}

\newcommand{\Z}{\mathbb Z}
\newcommand{\C}{\mathbb C}

\newcommand{\R}{\mathbb R}
\newcommand{\SL}{\operatorname{SL}}
\newcommand{\tr}{\operatorname{tr}}
\newcommand{\Tr}{\operatorname{tr}}

\newcommand{\e}{\mathrm e}
\newcommand{\ee}[1]{\e\!\left(#1\right)}

\newcommand{\rr}{\mathbf r}

\newcommand{\asf}{a}
\newcommand{\bsf}{b}
\newcommand{\csf}{c}
\newcommand{\dsf}{d}
\renewcommand{\dim}{d}

\newcommand{\smmattwoone}[2]
{\begin{psmallmatrix}
    #1 \\
    #2
\end{psmallmatrix}}
\newcommand{\smcoltwo}{\smmattwoone}

\newcommand{\smmattwo}[4]
{\begin{psmallmatrix}
    #1 & #2 \\
    #3 & #4
\end{psmallmatrix}}
\DeclareFontFamily{U}{rcjhbltx}{}
\DeclareFontShape{U}{rcjhbltx}{m}{n}{<->rcjhbltx}{}
\DeclareSymbolFont{hebrewletters}{U}{rcjhbltx}{m}{n}
\DeclareMathSymbol{\shin}{\mathord}{hebrewletters}{152}
\newcommand{\sfc}[3]{\shin^{#1}_{#2}\!\left(#3\right)}

\newcommand{\nin}{\notin}
\newcommand{\abs}[1]{\left| #1 \right|}
\newcommand{\Mod}[1]{\,\left(\operatorname{mod}\, #1\right)}
\newcommand{\ul}{\underline}
\newcommand{\mbf}{\mathbf}
\newcommand{\col}{{\rm col}}
\newcommand{\row}{{\rm row}}
\newcommand{\ct}{\dagger}
\newcommand{\sbe}[1]{|#1\rangle}

\title{The twisted convolution identity and ghost $r$-SICs from finite quantum dilogarithms}
\author{Marcus Appleby}
\address{University of Sydney}
\email{marcus.appleby@gmail.com}
\author{Steven T. Flammia}
\address{Virginia Tech and Phasecraft}
\email{stf@vt.edu}
\author{Gene S. Kopp}
\address{Louisiana State University}
\email{kopp@math.lsu.edu}
\date{September 30, 2026}

\begin{document}

\begin{abstract}
Radchenko and Wheeler (RW) recently proved a finite pentagon relation for real quadratic special values of the modular quantum dilogarithm and used it to establish the rank-$1$ twisted convolution identity conjectured by the current authors. RW gave an explicit argument in the principal case and remarked that their proof holds for all rank-$1$ admissible tuples. We extend their proof to all rank-$r$ admissible tuples and provide an explicit dictionary between the modular quantum dilogarithm and the Shintani--Faddeev modular cocycle conventions in the respective papers. Thus, we establish that, if $d,r$ are positive integers such that $r<\frac{d-1}{2}$ and $\frac{d^2-1}{r(d-r)} \in \mathbb{Z}$, then there exist ghost $r$-SICs: i.e., configurations of $d^2$ rank-$r$ subspaces in $\mathbb{C}^d$ that satisfy a non-Hermitian equichordal condition. Under the Stark conjecture, these configurations are Galois conjugate to Hermitian equichordal configurations called $r$-SICs (or rank-$r$ SIC-POVMs). 
\end{abstract}

\maketitle

\section{Introduction}\label{sec:introduction}

Zauner's conjecture~\cite{Zauner} postulates the existence of \(\dim^2\) complex equiangular lines in \(\C^\dim\) (also called SICs or SIC-POVMs in quantum information). 
The present authors (AFK) have constructed such lines~\cite{AFK} for every \(\dim > 3\), as well as analogous $r$-dimensional complex subspaces, using special values of the Shintani--Faddeev modular cocycle~\cite{Kopp}. 
Their passage from those special values to SICs is conditional on suitable Stark conjectures, which supply the
algebraicity and Galois switching, and on the \textit{twisted convolution identity} (TCI), a quadratic identity satisfied by the cocycle which ensures the constructed operator is idempotent. 
The construction requires an \textit{admissible tuple} as input, which is an ordered triple \(t = (\dim,r,Q)\) where the dimension \(\dim\), rank \(r\), and integer binary quadratic form \(Q\) must satisfy specific conditions (see Definition~\ref{def:admissible-tuple} below). 

Recently, Radchenko and Wheeler (RW) have proven the algebraicity of those same special values by finding explicit polynomial identities of the closely related modular quantum dilogarithm~\cite{RW}. 
They also proved the TCI with an argument written for the case \(r=1\) and where \(Q\) is a principal form~\cite[Thm.~7]{RW}; they remark that the proof holds for general forms and $r=1$, but is complicated to articulate owing to the intricate alignment needed between their conventions and those of \cite{AFK}. 

The goal of this paper is to provide an explicit extension of the RW proof of the TCI for every admissible tuple, including nonprincipal forms and all admissible ranks $r$, and to provide the dictionary required to map between the conventions of the two sets of authors. 

The proof we present is based on two main ideas.
\begin{itemize}
    \item[(1)] The RW pentagon relation implies a similar pentagon relation about sums of quotients of quantum dilogarithm values on dual pairs $H, H^\vee$ of subgroups of the RW metric group $G$. The proof of that implication, given in Section \ref{sec:subgroups}, is a straightforward application of the character theory of finite abelian groups. 
    \item[(2)] The AFK construction of $r$-SICs involves an odd-order symmetry that we denote by a matrix $L \in \SL_2(\Z)$ and whose order modulo $d$ we denote by $2m+1$. 
    The relations between $L$, the matrix $\gamma = L^{2m+1}$, and the admissible tuple are described in Section \ref{sec:sicsetup} and play a key role in the proof. The modular quantum dilogarithm (or Shintani--Faddeev cocycle) is preserved by an action of $L$, and that fact is crucial to obtaining the TCI in Section \ref{sec:tcc}.
\end{itemize}
The proof also requires a significant amount of translation between the conventions of AFK \cite{AFK, Kopp} and the conventions of RW \cite{RW}. A general translation between the expressions for the real quadratic special values used by the two sets of authors is carried out in Section \ref{sec:general}.

The results of this paper are unconditional and do not use any form of the Stark conjectures. 
The Stark conjectures enter only in passing from the ghost projections to a genuine SIC. 

\subsection{Main results}

We now state the main results of this paper. First, we provide a key definition.

\begin{definition}
\label{def:admissible-tuple}
    An \emph{admissible tuple} is a triple $t = (d,r,Q)$, where:
    \begin{enumerate}
        \item $d,r$ are positive integers such that $r < \frac{d-1}{2}$ and $\frac{d^2-1}{r(d-r)} \in \Z$;
        \item $Q$ is an integral binary quadratic form such that $\frac{(d+1)(d-3)}{\operatorname{disc}(Q)}$ is a square integer.
    \end{enumerate}
\end{definition}

\begin{theorem}\label{thm:tci}
    Let $t = (d,r,Q)$ be an admissible tuple, with $Q(x,y) = ax^2+bxy+cy^2$, let $L$ and $\gamma$ be defined in terms of $t$ as in Section \ref{sec:sicsetup}, and let $\tau$ be an attracting fixed point of $L$ under the fractional linear transformation action. 
    Let $\zeta_d = \e^{\frac{2\pi i}{d}}$, and let $\langle \mbf{x}, \mbf{y}\rangle = x_2y_1-x_1y_2$ for column vectors $\mbf{x},\mbf{y}\in\Z^2$.
    Let $\mbf{p}$ be a column vector in $\Z^2 \setminus \dim\Z^2$, and let $\mathcal{I}$ be a set of coset representatives of $\Z^2/\dim\Z^2$ containing $\mbf{0}$ and $\mbf{p}$.  
    Then there exists an integer $\lambda$ such that $2\lambda+\Tr(L)$ is coprime to $d$ and the Shintani--Faddeev cocycle $\shin$ satisfies the quadratic relation 
\begin{equation}\label{eq:tcc}
    \sum_{\mbf{q} \in \mathcal{I}} \zeta_{\dim}^{r\langle \mbf{p}, (\lambda I + L)\mbf{q} \rangle} \sfc{\dim^{-1}\mathbf{q}}{\gamma}{\tau} \sfc{\dim^{-1}(\mathbf{q}-\mathbf{p})}{\gamma^{-1}}{\tau} = 0.
\end{equation}
\end{theorem}

\begin{theorem}\label{thm:ghost}
    Let $\dim,r$ be positive integers such that $r < \frac{\dim-1}{2}$ and $\frac{\dim^2-1}{r(\dim-r)} \in \Z$. Then there exist $\dim$-by-$\dim$ complex matrices $\tilde\Pi_{\mbf{p}}$ indexed by $\mbf{p} \in \Z^2/\dim\Z^2$ satisfying the relations:
    \begin{itemize}
        \item[(1)] (Weyl--Heisenberg covariant) $\tilde\Pi_{\mbf{p}} = D_{\mbf{p}} \tilde\Pi_{\mbf{0}} D_{\mbf{p}}^{-1}$;
        \item[(2)] (Rank $r$ projections) $\tilde\Pi_{\mbf{p}}^2 = r\tilde\Pi_{\mbf{p}}$;
        \item[(3)] (Equichordal) $\Tr(\tilde\Pi_{\mbf{p}}\tilde\Pi_{\mbf{q}}) = \frac{r(\dim r-1)}{\dim^2-1}$ for $\mbf{p} \neq \mbf{q}$;
        \item[(4)] (Parity-Hermitian) $\tilde\Pi_{\mbf{p}}^\ct = U_P \tilde\Pi_{\mbf{p}} U_P^\ct$.
    \end{itemize}
    Here $D_{\mbf{p}}$ is the Weyl--Heisenberg displacement operator $D_{\mbf{p}} = \e^{\frac{(\dim+1)\pi i}{\dim}p_1p_2}X^{p_1}Z^{p_2}$ (a $\dim$-by-$\dim$ unitary matrix), with $X,Z$ are the standard shift and modulation operators acting on the standard basis by $X\sbe{j}=\sbe{j+1}$ and $Z\sbe{j}=\zeta^j\sbe{j}$, and $U_P$ is the partity operator acting on the standard basis by $U_P\,\sbe{j} = \sbe{-j}$, and the subscripts of the standard basis elements $\sbe{j}$ are interpreted modulo $\dim$.
\end{theorem}

The objects constructed by Theorem \ref{thm:ghost} are called \emph{ghost $r$-SICs}. The construction of $r$-SICs from admissible tuples conditional on the identity \eqref{eq:tcc} is described in detail in \cite{AFK}.

The construction of ($r$-)SICs from these ghost ($r$-)SICs relies on the existence of a Galois automorphism with certain properties. The existence of such a Galois automorphism is not known, but it is part of the prediction made by the Stark conjectures regarding the algebraic properties of the numbers $\shin_\gamma^{d^{-1}\mbf{p}}(\tau)$. (See \cite{AFK} and \cite{RW} for details.)

\subsection{Remarks on the proof of the finite pentagon relation}

The continuous pentagon relation for the modular quantum dilogarithm (first proved in the general case by Dimofte \cite{Dimofte}) is a Fourier transform identity, identifying a Fourier transform of a product of two modular quantum dilogarithms (times an Gaussian factor) as a product of three modular quantum dilogarithms (times a Gaussian factor). 
Within the span of one week in September 2026, three sets of authors (Huang \cite{Huang}, Rachenko and Wheeler \cite{RW}, and Gannon, Schopieray, and Yadav \cite{GSY}) posted preprints that prove finite, ``real multiplication value'' analogues of Dimofte's pentagon relation. Only Radchenko and Wheeler prove the general case or prove general algebraicity results. 
All three sets of authors prove a finite pentagon relation using a contour integration argument based on shifting the contour appearing in Dimofte's continuous pentagon relation (or a closely related identity) and summing the residues of the poles in a strip or a rectangular region (although the three proofs differ in some ways). 

In the authors' opinion, the finite pentagon relation may rightly be called a ``higher reciprocity law'' as part of a classical and venerable tradition. The sign of the quadratic Gauss sum, originally proven by Gauss in 1811, is an extension of quadratic reciprocity; it is a finite, ``integer value'' analogue of the Fourier transform formula for a Gaussian $\ee{\frac{p}{q}x^2}$. The proof of the finite pentagon relation is a (more complicated) higher analogue of a residue calculus proof of the evaluation of the quadratic Gauss sum, discovered by Kronecker \cite{Kronecker} and presented in the classic book of Apostol \cite[Theorem 9.16]{Apostol}.

Of course, the finite pentagon relation does not contain an important feature of Artin- or Shimura-style reciprocity laws:~It does not tell us how the Galois group acts. 
Algebraicity is instead obtained from the Ocneanu rigidity theorem for fusion categories, whose known proof is ineffective in the sense that it does not give an explicit description of the set of categorifications of a fusion ring, nor an explicit upper bound on the number of such categorifications. 
That is why, as of yet, these methods have not established Zauner's conjecture (a proof of which would require a Galois automorphism sending $\sqrt{\Delta} \to -\sqrt{\Delta}$  that intertwines with complex conjugation in a particular way), existence of unitary near-group categories (which would also require such a Galois automorphism), or the full rank-$1$ real quadratic Stark conjecture (which includes a complete description of the Galois action on the units).

\subsection{Notational conventions}

The papers \cite{Kopp,AFK} use column vectors exclusively, while the paper \cite{RW} uses row vectors for inputs to the function $F_\gamma^\pm$. The present paper translates between conventions of the two sets of authors. Hence, we make conventions explicit by writing column vectors as boldface lowercase letters $\mbf{x}$, row vectors as underlined lowercase letters $\ul{x}$, and modules of integer vectors as
\begin{align}
    \Z^2_{\col} &= \{\text{column vectors } \mbf{x} = \smcoltwo{x_1}{x_2} : x_1, x_2 \in \Z\}, \\
    \Z^2_{\row} &= \{\text{row vectors } \ul{x} = (x_1,x_2) : x_1, x_2 \in \Z\}.
\end{align}

In Section \ref{sec:general}, we use $\asf,\bsf,\csf,\dsf$ as entries of a two-by-two matrix $\gamma = \smmattwo{\asf}{\bsf}{\csf}{\dsf} \in \SL_2(\Z)$ with $c>0$ and $\Tr(\gamma)>2$. The matrix $\gamma$ is used throughout the paper and retains these properties, but its entries are only written explicitly in Section \ref{sec:general}.
Elsewhere in the paper, we use $a,b,c$ as coefficients of an integral binary quadratic form $Q$, and we use $\dim$ as the dimension of a vector space $\C^\dim$.

In Sections \ref{sec:general} and \ref{sec:subgroups}, we use $\varepsilon$ for the largest eigenvalue of the matrix $\gamma$, following the convention of \cite{RW}. 
In Sections \ref{sec:sicsetup} and \ref{sec:tcc}, we instead use $\varepsilon$ for the smallest totally positive algebraic unit $\varepsilon > 1$ in a real quadratic field $K$, following the convention of \cite{AFK}. 
The largest eigenvalue of $\gamma$ is then $\varepsilon^{j(2m+1)}$ in terms of this unit.

We use $\langle \ul{x}, \ul{y} \rangle_\gamma$ (taking row vectors, with the subscript $\gamma$) to denote a specific symmetric bicharacter, following the convention of \cite{RW}. We use $\langle \mbf{p}, \mbf{q} \rangle$ (taking column vectors, without a subscript) to denote a specific symplectic (not symmetric!) form, following the convention of \cite{AFK}.

We write $\ee{z} = \e^{2\pi i z}$ for the complex exponential function. We also write $N = \tr(\gamma)-2$ throughout.

\subsection*{Acknowledgments and AI use disclosure}

GSK thanks Xingting Wang and Qing Zhang for conversations about fusion categories that informed some expository content in this paper.

All material in the paper was written by the authors without AI assistance. 

\section{General translation between notations for real multiplication values of the modular quantum dilogarithm}\label{sec:general}

As a reminder to the reader, in this section only, $d$ is a matrix element, not a dimension. 

We now derive the relationship between the finite quantum dilogarithm $F^{\pm}_{\gamma}(\ul{x})$, which plays a central role in \cite{RW}, and the Shintani--Faddeev modular cocycle $\sfc{\mathbf{r}}{\gamma}{\tau}$, which plays a central role in \cite{AFK}.
\begin{proposition}\label{prop:Ftoshin}
    Let $\mathbf{u} \in \Z_{\col}^2$ and let $\gamma= \left(\begin{smallmatrix} \asf & \bsf \\ \csf & \dsf\end{smallmatrix}\right) \in \SL_2(\mathbb{Z})$ with $\Tr(\gamma)>2$ and $\csf>0$. 
    If $\mathbf{u} \nin (\gamma^{\top}-I)\Z_{\col}^2$, then
    \begin{align}
        F_\gamma^\pm(\mathbf{u}^\top) 
        &= \e^{\frac{\pi i}{12} \Psi(\gamma)}\,\sfc{\frac{1}{N}(I-\gamma^{-1})S\mathbf{u}}{\gamma}{\tau}^{-1},
    \end{align}
    while if $\mathbf{u} \in (\gamma^{\top}-I)\Z_{\col}^2$, then
    \begin{align}
        F_\gamma^\pm(\mathbf{u}^\top) 
        &= 
        \begin{cases}
            \left(\e^{- \frac{\pi i}{12} \Psi(\gamma)}\,\sfc{\frac{1}{N}(I-\gamma^{-1})S\mathbf{u}}{\gamma}{\tau}\right)^{\pm 1} \quad \text{if } \left((I-\gamma^{-1})S\mbf{u}\right)_2 > 0, \\
            \left(\e^{-\frac{\pi i}{12} \Psi(\gamma)}\,\sfc{\frac{1}{N}(I-\gamma^{-1})S\mathbf{u}}{\gamma}{\tau}\right)^{\mp 1} \quad \text{if } \left((I-\gamma^{-1})S\mbf{u}\right)_2  \leq 0.
        \end{cases}
    \end{align}
    Here $S=\left(\begin{smallmatrix} 0 & -1 \\ 1 & 0 \end{smallmatrix}\right)$, $N=\Tr(\gamma) -2$,   $\Psi(\gamma)$ is the Rademacher class invariant~\cite[Definition 1.29]{AFK}, and $\tau$ is the fixed point of $\gamma$ specified by 
    \begin{equation}
        \tau = \frac{\asf-\dsf + \sqrt{N(N+4)}}{2\csf}.
    \end{equation}.
\end{proposition}
\begin{proof}
Assume $\mathbf{u} \nin (\gamma^{\top}-I)\Z_{\col}^2$.
Radchenko and Wheeler define $F_\gamma^\pm$ in \cite[eq.~(2)]{RW} as
\begin{equation}
    F_\gamma^\pm(\mathbf{u}^\top) = \mu_\gamma\,\Phi_{\gamma,u_1,0}\left(\frac{u_1\tau+u_2}{\varepsilon^{-1}-1};\tau\right), \qquad \text{where }\  \varepsilon = \csf \tau + \dsf.
    \label{eq:FExpression}
\end{equation}
In this expression $\Phi$ is Fadeev's modular quantum dilogarithm, defined to be the analytic continuation of 
\begin{equation}
\Phi_{\gamma,m,n}(z;\tau) = \frac{\left(\ee{m\tau + z};\ee{\tau}\right)_\infty}{\left(\ee{n\frac{\asf\tau+\bsf}{\csf\tau+\dsf} + \frac{z}{\csf\tau+\dsf}};\ee{\frac{\asf\tau+\bsf}{\csf\tau+\dsf}}\right)_\infty}, \qquad \tau \in \C \setminus \R,
\end{equation}
where $(x,q)_\infty$ is the $q$-Pochhammer symbol and $\mu_\gamma$ is the multiplier system of the Dedekind eta function defined by
\begin{equation}
    \mu_{\gamma} = \frac{\eta(\gamma z)}{\eta(z)\sqrt{\csf z + \dsf}}
\end{equation}
for arbitrary $z$ in the upper half plane $\mathbb{H}$, and where  the principal branch of the square root is taken. 
The function $\Phi$ continues analytically  to $\tau \in \C_\gamma = \{\tau \in \C : j_\gamma(\tau) \nin (-\infty,0]\}$. 
(In \cite{AFK} the domain is denoted $\mathcal{D}_\gamma$.)
It follows from a result of Rademacher, which we stated as \cite[Proposition 5.1]{AFK}, that 
\begin{equation}
    \mu_\gamma = \e^{\frac{\pi i}{12} \Psi(\gamma)}
\end{equation}
where $\Psi(\gamma)$ is an explicit integer-valued class function called the \emph{Rademacher invariant} (or \emph{Rademacher symbol}).
Faddeev's modular quantum dilogarithm is basically the same thing as what in \cite{AFK,Kopp} is called the Shintani--Faddeev Jacobi cocycle.  To be precise, if the latter is specified via \cite[Definition 4.16]{Kopp} one has
\begin{equation}
    \Phi_{\gamma,m,n}(z;\tau) = \sigma_{\mathbf{v},\gamma}(m\tau+z,\tau)^{-1}
    \label{eq:moddilogTermsSFJacobi}
\end{equation}
where 
$\mathbf{v}=\smcoltwo{-m\bsf}{n-\dsf m}$. 
In the case of interest to us $m$ is  an integer so we can write instead $\mathbf{v}=\smcoltwo{0}{n-\dsf m}$. 
If instead we define $\sigma$ via \cite[Definition 1.16]{AFK} one has
\begin{equation}
    \Phi_{\gamma,m,n}(z;\tau) = 
    \frac{\left(\ee{\frac{z}{\csf \tau + \dsf}}, \ee{\frac{\asf\tau+\bsf}{\csf\tau+\dsf}}\right)_n}{(\ee{z},\ee{\tau})_m}\, \sigma_{\gamma}(z,\tau)^{-1}
\end{equation}
where the finite q-Pochhammer symbols in the prefactor are as specified by \cite[Definition 1.14]{AFK}, or \cite[eqs.~(48), (49)]{RW}.  The two definitions of $\sigma$ are related by $\sigma_{\gamma} = \sigma_{\boldsymbol{0},\gamma}$.

From 
$\varepsilon = \csf\tau+\dsf$, we have $\varepsilon^{-1} = -\csf\tau+\asf$, and $(\varepsilon-1)(\varepsilon^{-1}-1)=-N$, from which we obtain
\begin{align}
    \frac{1}{\varepsilon^{-1}-1} &= -\frac{\csf}{N}\tau + \frac{1-\dsf}{N};
    &
    \frac{\tau}{\varepsilon^{-1}-1} &= \frac{1-\asf}{N}\tau - \frac{\bsf}{N}.
\end{align}
Thus
\begin{equation}
    F_\gamma^\pm(\mathbf{u}^\top) = \e^{\frac{\pi i}{12} \Psi(\gamma)}\,\sigma_{\smcoltwo{0}{-du_1},\gamma}\!\left(\frac{(N+1-\asf)u_1 - \csf u_2}{N}\tau - \frac{\bsf u_1+(\dsf-1)u_2}{N}\right)^{-1}. 
\end{equation}
By \cite[Proposition 4.19]{Kopp}, we have the identity $\sfc{\rr}{A}{\tau} = \sigma_{(I-A)\rr,A}(r_2\tau-r_1,\tau)$. Also, if
\begin{equation}
    \rr = \frac{1}{N}\smcoltwo{\bsf u_1+(\dsf-1)u_2}{(N+1-\asf)u_1-\csf u_2}
    = u_1\mathbf{e}_2 + \frac{1}{N}(I-\gamma^{-1})S\mathbf{u},
\end{equation}
where $\mathbf{e}_2 = \smcoltwo{0}{1}$, then 
\begin{align}
    (I-\gamma)\rr 
    &= u_1(I-\gamma)\mathbf{e}_2 + \frac{1}{N}(2I-\gamma-\gamma^{-1})S\mathbf{u} \\
    &= u_1\smcoltwo{-\bsf}{1-\dsf} + \frac{1}{N}(-NI)S\smcoltwo{u_1}{u_2} \\
    &= \smcoltwo{-\bsf u_1}{(1-\dsf)u_1} + \smcoltwo{u_2}{-u_1} \\
    &= \smcoltwo{-\bsf u_1+u_2}{-\dsf u_1}.
\end{align}
The function $\sigma_{\mathbf{m},A}(z,\tau)$ does not depend on $m_1$ modulo $1$, so
\begin{align}
    F_\gamma^\pm(\mathbf{u}^\top) 
    &= \e^{\frac{\pi i}{12} \Psi(\gamma)}\, \sigma_{\smcoltwo{-\bsf u_1+u_2}{-\dsf u_1},\gamma}\!\left(\frac{(N+1-\asf)u_1 - \csf u_2}{N}\tau - \frac{\bsf u_1+(\dsf-1)u_2}{N}\right)^{-1} \\
    &= \e^{\frac{\pi i}{12} \Psi(\gamma)}\, \sfc{u_1\mathbf{e}_2 + \frac{1}{N}(I-\gamma^{-1})S\mathbf{u}}{\gamma}{\tau}^{-1} \\
    &= \e^{\frac{\pi i}{12} \Psi(\gamma)}\, \sfc{\frac{1}{N}(I-\gamma^{-1})S\mathbf{u}}{\gamma}{\tau}^{-1},
\end{align}
where we have also used the periodicity of $\shin$ in the the $\rr$ variable at RM values \cite[Proposition 4.35]{Kopp}. This completes the proof in the case when $\mathbf{u} \nin (\gamma-I)\Z_{\col}^2$.

In the case when $\mathbf{u} \in (\gamma^{\top}-I)\Z_{\col}^2$ (equivalently, when $\frac{1}{N}(I-\gamma^{-1})S\mathbf{u}\in \Z_{\col}^2$), the result may be deduced by comparing the definition $F_\gamma^\pm(\mathbf{u}^\top) := \varepsilon^{\pm 1/2}$ given in the sentence preceding \cite[eq.~(2)]{RW} to the evaluation of $\shin^\rr$ at integer values of $\rr$ given by \cite[Theorem 4.38]{Kopp}.
\end{proof}

\section{Finite pentagon relation for subgroups}\label{sec:subgroups}

We deduce a pentagon relation for subgroups of $G$ as a straightforward consequence of \cite[Theorem 2]{RW}; that relation will be used in Section \ref{sec:tcc} to prove the twisted convolution identity of \cite{AFK}. Before doing so, we recall the definitions of the lattice $\Lambda$, group $G$, and bicharacter $\langle \ul{g},\ul{h} \rangle_\gamma$ from \cite{RW}.

RW define $\Lambda = \Lambda_\gamma = N\Z_{\row}^2 + (N\Z_{\row}^2)(\gamma-I)^{-1}$ and $G = \Z^2/\Lambda$. (In fact, they write $\Lambda = N\Z_{\row}^2 + \ker(\gamma)$, with the notation $\ker(\gamma-1)$ denoting the ``mod $N$ kernel,'' that is, the preimage of $N\Z_{\row}^2$ under $\gamma-I$.) Using the identity $(\gamma-I)(I-\gamma^{-1}) = NI$, we see that $\Lambda = N\Z_{\row}^2 + \Z_{\row}^2(I-\gamma^{-1}) = \Z_{\row}^2(I-\gamma^{-1})$. Thus, the group $G = \Z_{\row}^2/\Z_{\row}^2(I-\gamma^{-1})$. (Because $\Z_{\row}^2\gamma = \Z_{\row}^2$, we can also write $\Lambda = \Z_{\row}^2(\gamma-I)$ and $G = \Z_{\row}^2/\Z_{\row}^2(\gamma-I)$.) The cardinality of this group is $\abs{G} = \abs{\det(\gamma-I)} = \abs{(\varepsilon-1)(\varepsilon^{-1}-1)} = \abs{2-\tr(\gamma)} = N$.

RW define a symmetric bicharacter $\langle \cdot, \cdot \rangle_\gamma : G \times G \to \C^\times$ by the formula
\begin{equation}
    \langle \ul x, \ul y \rangle_\gamma 
    = \ee{-\tfrac{1}{N} \ul x (\gamma-I)S \ul y^\top}.
\end{equation}
It is clearly a bicharacter ($\langle \ul x_1 + \ul x_2, \ul y \rangle_\gamma = \langle \ul x_1, \ul y \rangle_\gamma \langle \ul x_2, \ul y \rangle_\gamma$ and $\langle \ul x, \ul y_1 + \ul y_2 \rangle_\gamma = \langle \ul x, \ul y_1 \rangle_\gamma \langle \ul x, \ul y_2 \rangle_\gamma$). From the matrix identities
\begin{equation}
        (\gamma-I)S - \frac{N}{2}S = \left((\gamma-I)S\right)^{\!\top} + \frac{N}{2}S = \frac{1}{2}\left(\gamma S + (\gamma S)^\top\right),
\end{equation}
it may be seen that
\begin{equation}
    \langle \ul x, \ul y \rangle_\gamma 
    = \ee{-\tfrac{1}{N} \ul x (\gamma-I)S \ul y^\top}
    = \ee{-\tfrac{1}{N} \ul x ((\gamma-I)S)^\top \ul y^\top}
    = (-1)^{\ul x S \ul y^\top} \ee{-\tfrac{1}{N} \ul x (\gamma S + (\gamma S)^\top) \ul y^\top},
\end{equation}
and thus that $\langle \ul x, \ul y \rangle_\gamma$ is symmetric ($\langle \ul y, \ul x \rangle_\gamma = \langle \ul x, \ul y \rangle_\gamma$). It may be further checked directly that this bicharacter is nondegenerate (that is, if $\langle \ul x, \ul y \rangle_\gamma = 1$ for all $\ul x \in G$, then $\ul y = \ul 0$, and if $\langle \ul x, \ul y \rangle_\gamma = 1$ for all $\ul y \in G$, then $\ul x = \ul 0$).

We are now ready to deduce a pentagon relation for subgroups from the results of \cite{RW}. If the RW pentagon relation is thought of as an evaluation of a higher Gauss sum of sorts, the pentagon relation for subgroups is by analogy a higher Landsberg--Schaar relation.
(See, e.g., \cite{Moore} for the Landsberg--Schaar relation.)

\begin{theorem}\label{thm:subgrouppentagon}
    Fix notation as in \cite{RW}. Let $H$ be a proper subgroup of $G$, and let 
    \begin{equation}
    H^\vee = \{\ul{g} \in G : \langle \ul g,\ul h \rangle_\gamma = 1 \text{ for all } \ul h\in H\}.
    \end{equation}
    If $\ul y, \ul v \in G$, then
    \begin{align}
        \frac{1}{\sqrt{\abs{H}}} \sum_{\ul x \in H} F_\gamma^+(\ul x - \ul y) &= \frac{\mu_\gamma}{\sqrt{\abs{H^\vee}}} \sum_{\ul u \in H^\vee} \langle \ul y; \ul u\rangle_\gamma \frac{1}{F_\gamma^-(\ul u)}; \label{eq:subgrouppentagon1} \\
        \frac{1}{\sqrt{\abs{H}}} \sum_{\ul x \in H} \frac{F_\gamma^+(\ul x - \ul y)}{F_\gamma^-(\ul x - \ul y + \ul v)} &= \frac{1}{\sqrt{\abs{H^\vee}}F_\gamma^-(\ul v)} \sum_{\ul u \in H^\vee} \langle \ul y; \ul u\rangle_\gamma \frac{F^+(\ul u + \ul v)}{F^-(\ul u)} \quad \text{ for } \ul v \neq \ul 0. \label{eq:subgrouppentagon2}
    \end{align}
\end{theorem}
\begin{proof}
    The bicharacter $\langle\cdot,\cdot\rangle_\gamma$ defined is nondegenerate. As holds for a general non-degenerate bicharacter, the map $\ul g \mapsto \left(\ul h \mapsto \langle \ul g, \ul h \rangle_\gamma\right)$ defines an isomorphism from $G$ to its character group (or Pontryagin dual group) $\widehat{G}$. Under this isomorphism, $H^\vee$ is identified with the subgroup of characters vanishing on $H$. It follows that $H^{\vee\vee} = H$ and $\abs{H} \abs{H^\vee} = \abs{G}$.

    Note that 
    \begin{equation}
        \sum_{\ul u \in H^\vee} \langle \ul y; \ul u \rangle_\gamma \langle \ul x; \ul u \rangle_\gamma 
        = \sum_{\ul u \in H^\vee} \langle \ul x + \ul y; \ul u \rangle_\gamma
        = \abs{H^\vee} \delta_{\ul x+ \ul y}.
    \end{equation}
    Thus, summing \cite[eq.~(5)]{RW} over $\ul u \in H^\vee$ with coefficients $\langle \ul y; \ul u\rangle_\gamma$ immediately yields \eqref{eq:subgrouppentagon1}, and doing the same with \cite[eq.~(5)]{RW} immediately yields \eqref{eq:subgrouppentagon2}.
\end{proof}

\section{Admissible tuples and algebraic input data in the \texorpdfstring{$r$}{r}-SIC case}\label{sec:sicsetup}

We now specialize to an AFK admissible tuple. 
From Definition~\ref{def:admissible-tuple}, an admissible tuple is an ordered triple $(\dim,r,Q)$ satisfying certain conditions. 
Ref.~\cite{AFK} also gives a second equivalent definition in terms of alternative data, which we also need. 

Let $K$ be a real quadratic field with fundamental discriminant $\Delta_0$, and let $\varepsilon>1$ be its least totally positive unit.
For nonnegative integers $j,m$, define the sequence of conductors $f_j$, rank parameters $r_{j,m}$, and dimension parameters $d_{j,m}, d_j$ by
\begin{align}
    f_j &= \frac{\varepsilon^j-\varepsilon^{-j}}{\sqrt{\Delta_0}} \\
    r_{j,m} & = \frac{f_{jm}}{f_j} = \frac{\varepsilon^{jm}-\varepsilon^{-jm}}{\varepsilon^j-\varepsilon^{-j}} \label{eq:rjm} \\
    d_{j,m} &= r_{j,m+1}+r_{j,m} 
    = \frac{\varepsilon^{j(m+1)}-\varepsilon^{-jm}}{\varepsilon^j-1} 
    = \sum_{k=-m}^m \varepsilon^{jk} \label{eq:djm}\\
    d_j &= d_{j,1} = \varepsilon^j + 1 + \varepsilon^{-j}.
\end{align}
A primitive irreducible integral binary quadratic form $Q$ of discriminant $f^2\Delta_0$ is \emph{admissible} with these data when $f\mid f_j$.
We then define an \emph{alternative admissible tuple} as a quadruple
\[
 t=(K,j,m,Q)
\]
satisfying the above criteria. 

By~\cite[Thm.~1.25]{AFK}, admissible tuples $(\dim,r,Q)$ from Definition~\ref{def:admissible-tuple} are in canonical bijective correspondence with alternative admissible tuples $(K,j,m,Q)$ for the same quadratic form $Q$. 
We identify admissible tuples related by this bijection using the notation
\[
 t=(\dim,r,Q) \sim (K,j,m,Q).
\]
We drop the term ``alternative'' and refer to both types as simply \emph{admissible tuples} since they specify equivalent data and the type is clear from context. 

For a binary quadratic form $Q$, write $Q(x,y) = ax^2+bxy+cy^2$, and also use the symbol $Q$ to denote the half-Hessian matrix $Q = \smmattwo{a}{\frac{b}{2}}{\frac{b}{2}}{c}$. 
Following~\cite[Definition 1.28]{AFK}, we define certain stabilizers of a quadratic form $Q$ associated to an admissible tuple. 
With the admissible tuple understood from context, define 
\[
L = 
\frac{d_j-1}{2}I + \frac{f_j}{f}SQ 
\quad \text{and }\ A 
= L^{2m+1}.
\]
In the notation of \cite{RW}, we will take $\gamma = A$ and set $N = \Tr(\gamma)-2$.
For convenience and agreement with \cite{RW}, we assume that $\gamma_{21}>0$; it is always possible to replace $Q$ by an $\SL_2(\Z)$-equivalent quadratic form to make that true.

We prove several short, elementary lemmas about these parameters.
\begin{lemma}\label{lem:N}
    $N = (d_j-3)d_{j,m}^2.$
\end{lemma}
\begin{proof}
The eigenvalues of $L$ are $\varepsilon^{\pm j}$. Thus,
$
    \Tr(\gamma) = \varepsilon^{(2m+1)j} + \varepsilon^{-(2m+1)j}.
$
The identity $d_{j,m} = \frac{\varepsilon^{j(m+1)}-\varepsilon^{-jm}}{\varepsilon^j-1}$ follows from \eqref{eq:rjm} and \eqref{eq:djm}. Using that identity, we calculate
\begin{align}
    (d_j-3)d_{j,m}^2
    &= \varepsilon^{(2m+1)j} - 2 + \varepsilon^{-(2m+1)j} \quad \mbox{(after some algebra)} \\
    &= \Tr(\gamma) - 2 = N. \qedhere
\end{align}
\end{proof}

\pagebreak

\begin{lemma}\label{lem:Lpower}
    $L^{2m+1}-I = d_{j,m} L^m(L-I)$.
\end{lemma}
\begin{proof}
Using the fact that $B+B^{-1} = \Tr(B)I$ for any $B \in \SL_2(\Z)$, we have
\begin{align}
    L^{2m+1}-I
    &= \left(\sum_{k=-m}^m L^k\right)L^m(L-I) \\
    &= \left(1 + \sum_{k=1}^m \Tr(L^k)\right)L^m(L-I) \\
    &= \left(1 + \sum_{k=1}^m (\varepsilon^{jk} + \varepsilon^{-jk})\right)L^m(L-I) \\
    &= d_{j,m} L^m(L-I) \quad \text{by \eqref{eq:djm}}. \qedhere
\end{align}
\end{proof}

\begin{lemma}\label{lem:Ltom}
    For integers $j,m \geq 1$, $L^m = r_{j,m}L - r_{j,m-1}I$. (Here we set $r_{j,0}=0$.)
\end{lemma}
\begin{proof}
    Straightforward induction using the quadratic characteristic polynomial $L^2 - r_{j,2}L + I = 0$ (where $r_{j,2} = \varepsilon + \varepsilon^{-1} = \Tr(L)$) and the formula \eqref{eq:rjm}.
\end{proof}

\begin{lemma}\label{lem:detLminusI}
    $\det(L-I) = -(d_j-3)$.
\end{lemma}
\begin{proof}
    The eigenvalues of $L-I$ are $\varepsilon^j-1$ and $\varepsilon^{-j}-1$, and their product is $(\varepsilon^j-1)(\varepsilon^{-j}-1) = 2-\varepsilon^j-\varepsilon^{-j} = 3-d_j$.
\end{proof}

\section{Proof of the twisted convolution identity and ghost \texorpdfstring{$r$}{r}-SIC existence}\label{sec:tcc}

We now prove the main results, which are Theorem \ref{thm:tci} and Theorem \ref{thm:ghost}.

\begin{proof}[Proof of Theorem \ref{thm:tci}]
Set $H = d_{j,m}\Z_{\row}^2/\Z_{\row}^2(\gamma-I)$, so that we obtain $H^\vee = \Z_{\row}^2(L-I)/\Z_{\row}^2(\gamma-I)$. 
Here $\abs{H} = d_j-3$ and $\abs{H^\vee} = d_{j,m}^2$.

Choose $\ul v \in H^\vee \setminus \{0\}$. Then there exists some $\ul y \in G \setminus H$ such that $\ul v = -\ul{y}(L-I) = \ul{y}-\ul{y}L$. Theorem \ref{thm:subgrouppentagon} (specifically \eqref{eq:subgrouppentagon2}) gives
\begin{equation}\label{eq:tccproof1}
    \frac{1}{\sqrt{d_j-3}}\sum_{\ul x \in H} \frac{F_\gamma^+(\ul{x}-\ul{y})}{F_\gamma^-(\ul{x}-\ul{y}L)}
    = \frac{1}{d_{j,m}F_\gamma^-(\ul v)} \sum_{\ul{u} \in H^\vee} \langle \ul{y}; \ul{u}\rangle_\gamma \frac{F^+(\ul{u}+\ul{v})}{F^-(\ul{u})}.
\end{equation}

By \cite[Proposition 8]{RW} (or as a consequence of \cite[Theorem 4.37]{Kopp}), we have the order-$(2m+1)$ symmetry $F_\gamma^{\pm}(\ul{z}L) = F_\gamma^{\pm}(\ul{z})$ for all $\ul{z} \in G$. (This is a generalization of the order-$3$ Zauner symmetry for $1$-SICs.) Note also that $\ul{z}L = \ul{z}$ for all $\ul{z} \in H$. Thus, for $\ul{x} \in H$, $\ul{x}-\ul{y} \neq \ul{0}$ (because $\ul y \nin H$), and
\begin{equation}
    \frac{F_\gamma^+(\ul x - \ul y)}{F_\gamma^-(\ul x - \ul y L)}
    = \frac{F_\gamma^-(\ul x - \ul y)}{F_\gamma^-(\ul x - \ul y L)}
    = \frac{F_\gamma^-((\ul x - \ul y)L)}{F_\gamma^-(\ul x - \ul y L)}
    = \frac{F_\gamma^-(\ul x L - \ul y L)}{F_\gamma^-(\ul x - \ul y L)}
    = \frac{F_\gamma^-(\ul x - \ul y L)}{F_\gamma^-(\ul x - \ul y L)}
    = 1.
\end{equation}
Therefore, the left-hand side of \eqref{eq:tccproof1} simplifies to $\sqrt{d_j-3}$, and thus we may rewrite \eqref{eq:tccproof1} as
\begin{equation}
    -d_{j,m}\sqrt{d_j-3}F_\gamma^-(\ul v)\,
    + \sum_{\ul u \in H^\vee} \langle \ul y; \ul u\rangle_\gamma \frac{F^+(\ul u + \ul v)}{F^-(\ul u)} = 0.
\end{equation}
Taking the term where $\ul u = \ul 0$ out of the sum, and noting that $F_\gamma^-(0) = \varepsilon^{-(2m+1)j/2}$ and $F_\gamma^-(\ul u) = F_\gamma^+(\ul u)$ for $\ul u \neq \ul 0$ (including $\ul u = \ul v$), we have
\begin{equation}
    \left(\varepsilon^{(2m+1)j/2}-d_{j,m}\sqrt{d_j-3}\right)F_\gamma^+(\ul v) \,
    + \sum_{\ul u \in H^\vee \setminus \{\ul 0\}} \langle \ul y; \ul u\rangle_\gamma \frac{F^+(\ul u + \ul v)}{F^+(\ul u)} = 0.
\end{equation}
But $\varepsilon^{(2m+1)j/2}-d_{j,m}\sqrt{d_j-3} = \varepsilon^{-(2m+1)j/2} = \frac{1}{F_\gamma^+(\ul v)}$, and therefore
\begin{equation}\label{eq:tccproof2}
    \sum_{\ul u \in H^\vee} \langle \ul y; \ul u\rangle_\gamma \frac{F_\gamma^+(\ul u + \ul v)}{F_\gamma^+(\ul u)} = 0.
\end{equation}

Now write $\ul u^\top = \mathbf{u} + (\gamma-I)\Z^2$, $\ul v^\top = \mathbf{v} + (\gamma-I)\Z^2$, and $\ul v^\top = \mathbf{v} + (\gamma-I)\Z^2$, where $\mbf{u} = \smcoltwo{u_1}{u_2}, \mbf{v} = \smcoltwo{v_1}{v_2}, \mbf{y} = \smcoltwo{y_1}{y_2}$ are column vectors in $\Z_{\col}^2$.
By Proposition \ref{prop:Ftoshin}, 
if $u_1, u_1+v_1 \nin \Z^+$, then
\begin{align}\label{eq:Ffracshinfrac1}
    \frac{F_\gamma^+(\ul{u}+\ul{v})}{F_\gamma^+(\ul{u})} 
    &= \frac{\sfc{\frac{1}{N}(I-\gamma^{-1})S\mathbf{u}}{\gamma}{\tau}}{\sfc{\frac{1}{N}(I-\gamma^{-1})S(\mathbf{u}+\mathbf{v})}{\gamma}{\tau}}.
\end{align}
We have the $2 \times 2$ matrix identities
\begin{align}
    \frac{1}{N}(I-\gamma^{-1}) S (L-I)^\top 
    &= \frac{1}{N}L^{-2m-1}(L^{2m+1}-I) S (L-I)^\top \\
    &= \frac{1}{d_{j,m}^2(d_j-3)}L^{-2m-1}\left(d_{j,m} L^m(L-I)\right) S (L-I)^\top \quad \text{by Lemmas \ref{lem:N} and \ref{lem:Lpower}} \\
    &= \frac{1}{d_{j,m}(d_j-3)}L^{-m-1} \left((L-I) S (L-I)^\top\right) \\
    &= \frac{1}{d_{j,m}(d_j-3)}L^{-m-1} \left(\det(L-I) S\right) \\
    &= -\frac{1}{d_{j,m}}L^{-m-1}S \quad \text{ by Lemma \ref{lem:detLminusI}}.
\end{align}
We have $\mbf{v} = -(L-I)^\top\mbf{y}$. If we also express $\mbf{u} = (L-I)^\top\mbf{x}$ for some $\mbf{x} \in \Z_{\col}^2$, then \eqref{eq:Ffracshinfrac1} becomes
\begin{align}\label{eq:Ffracshinfrac2}
    \frac{F_\gamma^+(\ul{u}+\ul{v})}{F_\gamma^+(\ul{u})} 
    &= \frac{\sfc{\frac{1}{N}(I-\gamma^{-1})S(L-I)^\top\mathbf{x}}{\gamma}{\tau}}{\sfc{\frac{1}{N}(I-\gamma^{-1})S(L-I)^\top(\mathbf{x}-\mathbf{y})}{\gamma}{\tau}}
    = \frac{\sfc{-d_{j,m}^{-1}L^{-m-1}S\mathbf{x}}{\gamma}{\tau}}{\sfc{-d_{j,m}^{-1}L^{-m-1}S(\mathbf{x}-\mathbf{y})}{\gamma}{\tau}}.
\end{align}
Additionally, the bicharacter becomes
\begin{align}
    \langle \ul y; \ul u \rangle_\gamma
    &= \ee{-\tfrac{1}{N} \mbf{y}^\top (\gamma-I)S \mbf{u}} \\
    &= \ee{-\mbf{y}^\top \gamma \left(\tfrac{1}{N} (I-\gamma^{-1})S (L-I)^\top\right)\mbf{x}} \\
    &= \ee{-\mbf{y}^\top L^{2m+1} \left(-\tfrac{1}{d_{j,m}} L^{-m-1}S\right)\mbf{x}} \\
    &= \ee{\tfrac{1}{d_{j,m}}\mbf{y}^\top L^{m}S\mbf{x}}.
\end{align}
Now express $\mbf{x} = SL^{m+1}\mbf{q}$ and $\mbf{y} = SL^{m+1}\mbf{p}$. We obtain
\begin{align}\label{eq:Ffracshinfrac3}
    \frac{F_\gamma^+(\ul{u}+\ul{v})}{F_\gamma^+(\ul{u})} 
    &= \frac{\sfc{d^{-1}\mathbf{q}}{\gamma}{\tau}}{\sfc{d^{-1}(\mathbf{q}-\mathbf{p})}{\gamma}{\tau}}
    = \sfc{d^{-1}\mathbf{q}}{\gamma}{\tau} \sfc{d^{-1}(\mathbf{q}-\mathbf{p})}{\gamma^{-1}}{\tau}
\end{align}
(using the cocycle condition in the last step), and
\begin{align}\label{eq:bicharacterfinal}
    \langle \ul y; \ul u \rangle_\gamma
    &= \ee{-\tfrac{1}{d_{j,m}}\mbf{y}^\top (L^{m+1})^\top S^\top L^{2m+1}\mbf{x}}
    = \ee{\tfrac{1}{d_{j,m}}\mbf{y}^\top S L^{m}\mbf{x}}
    = \zeta_{d_{j,m}}^{\langle \mbf{p}, L^m\mbf{q} \rangle}.
    = \zeta_{d_{j,m}}^{\langle \mbf{p}, (r_{j,m}L - r_{j,m-1}I)\mbf{q} \rangle}
\end{align}
(using Lemma \ref{lem:Ltom} in the last step).

Let $\mathcal{I}$ be a set of coset representatives of $\Z_\col^2/d_{j,m}\Z_\col^2$ containing $\mbf{0}$ and $\mbf{p}$. Substituting \eqref{eq:Ffracshinfrac3} and \eqref{eq:bicharacterfinal} into \eqref{eq:tccproof2}, and taking $d=d_{j,m}$, $r=r_{j,m}$, and $\lambda$ an integer satisfying \mbox{$\lambda r \equiv -r_{j,m-1} \Mod{d}$}, we obtain
\begin{equation}
    \sum_{\mbf{q} \in \mathcal{I}} \zeta_{d}^{r\langle \mbf{p}, (\lambda I + L)\mbf{q} \rangle} \sfc{d^{-1}\mathbf{q}}{\gamma}{\tau} \sfc{d^{-1}(\mathbf{q}-\mathbf{p})}{\gamma^{-1}}{\tau} = 0,
\end{equation}
which is the twisted convolution identity conjectured in \cite{AFK} (for the choice of shift $\lambda$).
\end{proof}

\begin{proof}[Proof of Theorem \ref{thm:ghost}]
    Follows from Theorem \ref{thm:tci} and \cite[Theorem 1.45]{AFK}.
\end{proof}



\end{document}